\documentclass{amsart}
\usepackage{amsfonts,amssymb,amsmath,amsthm,graphicx,color,diagmac,tikz,tikz-cd,soul}
\usepackage{url}
\usepackage[normalem]{ulem}
\usepackage{enumerate}

\newtheorem{thrm}{Theorem}[section]
\newtheorem{lem}[thrm]{Lemma}
\newtheorem{prop}[thrm]{Proposition}
\newtheorem{cor}[thrm]{Corollary}
\theoremstyle{definition}

\newtheorem{remark}[thrm]{Remark}
\newtheorem{ex}[thrm]{Example}
\numberwithin{equation}{section}

\newcommand{\bte}{\begin{theorem}\quad  }
\newcommand{\ete}{\end{theorem} }
\newcommand{\bpr}{\begin{prop}\quad  }
\newcommand{\epr}{\end{prop} }
\newcommand{\ble}{\begin{lemma}\quad }
\newcommand{\ele}{\end{lemma}}
\newcommand{\bco}{\begin{coro}\quad }
\newcommand{\eco}{\end{coro} }
\newcommand{\bex}{\begin{exa}\quad \rm }
\newcommand{\eex}{\end{exa} }
\newcommand{\bde}{\begin{defi}\quad \rm }
\newcommand{\ede}{\end{defi} }
\newcommand{\brm}{\begin{rem} \quad \rm}
\newcommand{\erm}{\end{rem} }
\newcommand{\bpf}{\begin{proof}[\bf{Proof.\quad}] \rm}
\newcommand{\epf}{ \end{proof}}
\newcommand{\bdm}{\begin{displaymath} }
\newcommand{\edm}{\end{displaymath} }
\newcommand{\be}{\begin{eqnarray*}}
\newcommand{\ee}{\end{eqnarray*}  }

\newcommand{\ba}{\begin{align*}}
\newcommand{\ea}{\end{align*}}

\makeatletter
\newcommand\cupop{\mathop{\operator@font \cup}\nolimits}
\makeatother
\numberwithin{equation}{section}

\author{Mojtaba Sedaghatjoo}
\address{
Department of Mathematics, College of Sciences, Persian Gulf University, Bushehr, Iran.}

\author{Mohammad Ali Naghipoor $^*$}
\address{
Department of Mathematics, Jahrom University, PO Box 74135-111, Jahrom, Iran.}
\email{ma\underline{ }naghipoor@yahoo.com}
\thanks{$*$Corresponding author}

\begin{document}

\title[Injectivity of Acts Based on Indecomposability]{An Approach to Injective Acts over Monoids Based on Indecomposability}

\begin{abstract}
The main purpose of this paper is investigating classes of acts that are injective relative to all embeddings with indecomposable domains or codomains. We give some homological classifications of monoids in light of such kinds of injectivity. Our approach to indecomposable property provides a new characterization  of right absolutely injective monoids as ones that all indecomposable acts are injective.
\end{abstract}
\maketitle

\section{\bf Introduction}
\noindent
Throughout this paper, unless otherwise stated, $S$ is a monoid in which $1$ denotes its identity element. A right $S-$act $A_S$ is a non-empty set on which $S$ acts unitarily. By an act or a $S-$act we simply mean a right $S-$act for a monoid $S.$ The one element act is called zero act and is denoted by $\Theta_S$. As for each monoid $S$ we can associate the monoid $S^0$ with a zero element, for each right $S$-act $A$ we can associate the right $S$-act $A^{\theta}_{S}$ with a zero element as below: \[ A^{\theta}_{S} = \begin{cases} A_S & \text{if } A_S \text{ contains a zero element,} \\ A_S \cup \Theta & \text{otherwise.}   \end{cases} \]

A monoid $S$ is called \textit{left reversible} if every two right ideals of $S$ have a non-empty intersection, that is, $aS\cap bS\neq\emptyset,$ for each $a, b\in S.$

We mean by $A\sqcup B$ the disjoint union of  sets $A$ and $B$. A right $S$-act $A$ is called \textit{decomposable} provided that there exist subacts $B, C\subseteq A$ such that $A=B \cup C$ and $B \cap C= \emptyset$. In this case $A=B \cup C$ is called a decomposition of $A$. Otherwise $A$ is called an \textit{indecomposable} right $S-$act. For a nonempty set $I$, $S^I$ denotes the set $\prod \limits _I S$, endowed with the natural componentwise right $S$-action ${(s_i)}_{i\in I}s={(s_is)}_{i\in I}$.
\begin{remark} \label{ind defn}
It is well-known that every $S$-act $A$ has a unique decomposition into indecomposable subacts. Indeed, indecomposable components of $A_S$ are precisely the equivalence classes of the relation $\sim $ on $A_S$  defined by $a\sim b$ if there exist $s_1,s_2, \ldots ,s_n,t_1,t_2, \ldots ,t_n \in S,~a_1,a_2, \ldots , a_n \in A_S$ such that
\[ a=a_1s_1,~ a_1t_1=a_2s_2,~ a_2t_2=a_3s_3, \ldots,~a_nt_n=b \] (see \cite{Ren}).  Therefore, elements $a,b\in A_S$ are in the same indecomposable component if and only if there exists a sequence of equalities of length $n$  as above connecting $a$ to $b$.
In the Proposition \ref{pr6} we will show that for left reversible monoids the length of the foregoing sequence can be taken 2. Note that from this aspect it is straightforward to show that homomorphic images of indecomposable acts are indecomposable (see \cite[Lemma 1.5.36] {kilp}).
\end{remark}
Recall that for a nonempty set $I$, $I^S$ is an \textit{$|I|$-cofree right $S$-act} where $fs$ for $f\in I^S, ~s\in S$ is defined by $fs(t)=f(st)$ for every $t\in S$. It should be mentioned that the 1-cofree object or terminal object in {\bf Act-}$S$ is the one element act  which is indecomposable.

A right $S-$act $A$ is called an \textit{injective right $S-$act}, if for any $S-$act $B,$ any subact $C$ of $B$ and any homomorphism $f\in \text{Hom}(C,A),$ there exists a homomorphism $\overline{f}\in \text{Hom}(B,A)$ extending $f,$ that is, $\overline{f}\mid_C=f.$


A monoid $S$ over which all right $S-$acts are injective is called an \textit{absolutely (completely) right injective monoid}. The study of such monoids was initiated by A. Skornjakov in \cite{Skornjakov69} and  E. Feller and L. Gantos in \cite{Feller1},\cite{Feller2} and \cite{Feller3}, though J. Fountaim in \cite{Foun} presented an account on characterizing such monoids in general situations. A later work in this issue was done by K. Shoji in \cite{Shoji}. Injectivity of acts over monoids for special classes of homomorphisms and for special classes of embeddings has been studied by many authors that leads to different weakened forms of injectivity. For instance, \cite{Ebrahimi}, \cite{Gould}, \cite{Zhang} and \cite{Zhang2} are papers ranging over these problems. By Skornjakov-Baer criterion, for checking injectivity of any right act with a zero element, it suffices to consider all inclusions into \textbf{cyclic} acts in diagram \textbf{(I)}. Besides if $C$ in the above diagram is cyclic then $A$ is called \textit{C-injective}, which is studied in \cite{Zhang}. If in addition $B$ is cyclic in diagram \textbf{(I)}, then $A$ is called \textit{CC-injective} which is introduced in \cite{Zhang2}. They study conditions under which all cyclic acts are injective. As every cyclic act is indecomposable and also every act is decomposed into indecomposable ones, it seems reasonable to ask whether similar results hold whenever $C_S$ or $B_S$ are indecomposable in diagram \textbf{(I)}. This paper is devoted to investigating the issue for classes of monomorphisms with indecomposable domains and codomains by which we lead to an approach to absolutely right injective monoids. 

To introduce the new kinds of injectivity we follow diagram (I) of injectivity and define the following notions of injectivity. 

\begin{enumerate}[(i)]
\item  A right $S-$act $A$ is called \textit{indecomposable codomain injective} or \textit{InC-injective} for short, if it is injective relative to all embeddings into indecomposable acts. 

\item A right $S-$act $A$ is called \textit{indecomposable domain injective} or \textit{InD-injective}, if it is injective relative to all embeddings from indecomposable acts. 

\item A right $S-$act $A$ is called  \textit{pseudo indecomposable domain injective} or \textit{PInD-injective} for short, if for any right $S-$act $B,$ any indecomposable right subact $C$ of $B$ and any \textbf{monomorphism} $f\in \text{Hom}(C,A),$ there exists a homomorphism $\overline{f}\in \text{Hom}(B,A)$ which extends $f,$ that is, $\overline{f}\mid_C=f.$


\end{enumerate}

We prove that InC-injective acts are those ones which are injective relative to all embeddings into cyclic acts. Besides, Theorem \ref{InC-inj+0=inj} states that injective acts are the same InC-injective acts with zero. As cyclic acts are indecomposable, we have the implications:

$$\text{injectivity} \Longrightarrow \text{InD-injectivity} \Longrightarrow \text{C-injectivity}.$$

all are strict.

Since every indecomposable act can be embedded in an $S-$act with a zero, it is easy to see that every InD-injective (PInD-injective) right $S-$act contains a zero. In Theorem \ref{completely InD-injective monoid}, we give a characterization of monoids over which all acts are InD-injective (PInD-injective). We show that these monoids are the same as absolutely injective monoids.

The reader is referred to \cite{clifford} and \cite{kilp} for preliminary account on the needed notions and results related to act and semigroup theory.

\section{Indecomposable Codomain Injective Acts}

Regarding the main role of indecomposable acts in this paper we shall take them into consideration more precisely.

First note that if $A$ is an indecomposable right $S-$act contained in $\bigsqcup_{i\in I}B_i,$ a coproduct of a family of right $S-$acts $\{B_i\}_{i\in I},$ then $A\subseteq B_i,$  for some $i\in I.$

\begin{prop} \label{pr6}
Let $S$ be a left reversible monoid. Then a right $S-$act $A$ is indecomposable if and only if for any $a,a'\in A_S$ there exist $s,s' \in S$ such that $as=a's'$.
\end{prop}

\begin{proof}
Let $S$ be a left reversible monoid. Suppose that $A_S$ is indecomposable and $a,a'\in A$. Using Remark \ref{ind defn}, there exists a sequence of equalities, connecting $a$ to $a'$, of the form \[ a=a_1s_1,-~ a_1t_1=a_2s_2,~ a_2t_2=a_3s_3, \ldots,~a_nt_n=a',\]  for $a_i\in A_S,~s_i,t_i\in S,~1\leqslant i \leqslant n $. Left reversibility of $S$ provides  $u_1,u_2 \in S$ such that $s_1u_1=t_1v_1,$ and, in consequence, $au_1=a_1s_1u_1=a_1t_1v_1=a_2s_2v_1$. Proceeding inductively, we get $u,v\in S$ providing $au=a_nt_nv=a'v$ as desired. The converse directly follows by Remark \ref{ind defn}.
\end{proof}

Considering the fact that a monoid $S$ is left reversible if and only if for any $u, u' \in S$ there exist $s,s' \in S$ such that $us=u's'$ Proposition \ref{pr6} says that indecomposable acts over such monoids play a counterpart role of $S$ for acts.

\begin{prop} \label{pr7}
For a monoid $S$ all subacts of indecomposable right $S$-acts are indecomposable if and only if $S$ is left reversible.
\end{prop}

\begin{proof} Necessity. Let $a,b\in S$. Since $S$ is indecomposable, our assumption implies that $aS\cup bS$ is indecomposable and therefore $aS \cap bS \neq \emptyset$.

Sufficiency. This is a straightforward application of Proposition \ref{pr6}.
\end{proof}

The next proposition characterizes monoids over which non-zero cofree acts are decomposable.

\begin{prop} \label{pr8} For a monoid $S$ the following are equivalent.
\begin{itemize}
\item[{\rm (i)}] All non-zero cofree $S-$acts are decomposable,

\item[{\rm (ii)}] there is a non-zero decomposable cofree $S-$act,

\item[{\rm (iii)}] $S$ is left reversible.
\end{itemize}
\end{prop}
\begin{proof}
 $i\Longrightarrow ii$ is clear.

$ii\Longrightarrow iii$. By way of contradiction suppose that $aS \cap bS=\emptyset$ for some $a,b\in S$. Let $X^S$ be a non-zero decomposable $|X|$-cofree act and $f,g \in X^S$.  Let $h \in X^S$ be given by

$$h(x)=  \begin{cases} f(x) & \text{if } x\in aS, \\ g(x) &  \text{otherwise.}
\end{cases}$$

So we get the sequence $f=f.1,~fa=ha,~hb=gb,~g.1=g$, which implies that $f$ and $g$ are in the same indecomposable component. Therefore $X^S$ is indecomposable a contradiction.

$iii\Longrightarrow i$. Let $S$ be a left reversible monoid and $X^S$ be a non-zero cofree $S$-act. Take constant functions $f=c_{x_1}$ and $g=c_{x_2}$ in $X^S$ for different elements $x_1$ and $x_2$ in $X$. Then $f$ and $g$ are zero elements of $X^S$ (indeed zero elements of $X^S$ are the same constant functions). If $f$ and $g$ are in the same indecomposable component, in light of Proposition \ref{pr6}, there exist $a,b \in S$ such that $fa=gb$ and so $f=g$, a contradiction.
\end{proof}

\begin{lem}\label{le2}
If $Q_S=Q_1 \sqcup Q_2$ contains a zero element, then $Q_1^{\theta}$ is a retract of $Q_S$.
\end{lem}

\begin{proof}
Since $Q_S$ contains a zero element as $\theta$, $Q_1^{\theta} \subseteq Q_S$. Now for the epimorphism $f: Q_S \to Q_1^{\theta}$ given by

$$f(x)=\begin{cases} x & \text{if } x\in Q_1^{\theta} \\ \theta & \text{otherwise.} \end{cases}$$

we get $f \circ i=\text{id}_{Q_1^{\theta}}$ where $i$ is the inclusion map from $Q_1^{\theta}$ to $Q_S$.
\end{proof}

The next result immediately follows by the above lemma.
\begin{lem}\label{le3}
If $Q=Q_1 \sqcup Q_2$ is an injective $S-$act, then $Q_1^{\theta}$ is injective too.
\end{lem}

 In view of the foregoing lemma,  we conclude that  if $Q=Q_1 \sqcup Q_2$ is an injective $S-$act, then either $Q_1$ or $Q_2$ is injective. 

The next proposition gives a characterization of InC-injective acts which proves those are very close to injective acts.

\begin{thrm} \label{InC-inj+0=inj}
Let $S$ be a monoid. An $S-$act $Q$ is InC-injective if and only if $Q^{\theta}_{S}$ is injective.
\end{thrm}

\begin{proof}
First suppose that $Q$ is InC-injective right $S-$act. If $Q_S$ contains a zero it is injective in view of Skornjakov-Baer criterion for acts (\cite[Theorem 1]{Skornjakov69}) and there is nothing to prove. Otherwise, $Q_S$  contains no zero and hence $S$ is left reversible. Indeed, if $S$ is not left reversible there exist right ideals $I$ and $J$ in $S$ for which $I\cap J=\emptyset$. Now consider the following diagram
\begin{center}
\begin{tikzcd}
I \arrow{r}{\subseteq} \arrow{d}[swap]{f}
&S/J \\
Q &
\end{tikzcd}
\end{center}
where $S/J$ is the Rees factor act of $S_S$ by the right ideal $J$ and $f:I\longrightarrow Q_S$ is given by $f(i)=qi,i\in I$ for a fixed element $q\in Q_S$. Now our assumption provides a homomorphism $\bar{f} \in \text{Hom}(S/J, Q)$ commuting the above diagram which yields a zero in $Q_S$, a contradiction. Now regarding Skornjakov-Baer criterion for injective acts, consider the following diagram
 \begin{center}
\begin{tikzcd}
A \arrow{r}{\subseteq} \arrow{d}[swap]{f}
&B=S/\rho \\
Q^{\theta} &
\end{tikzcd}
\end{center}

where $\rho$ is a right congruence on $S$, $B$ is a subact of a cyclic $S-$act $B=S/\rho$ and $f$ is a homomorphism.  Since $Q_S$ contains no zero $Q_S^{\theta}=Q_S\cup \Theta$. As $S$ is left reversible and $S/\rho$ is indecomposable, $A_S$ is indecomposable by Proposition \ref{pr7} and thus ${\rm{Im}}f\subseteq Q$ or ${\rm{Im}}f \subseteq \Theta$. Therefore in both cases by our assumption we have desired morphism for making the diagram commutative.

Conversely, suppose that $Q^{\theta}_{S}$ be an injective act and consider the following diagram,
\begin{center}
\begin{tikzcd}
A \arrow{r}{\iota} \arrow{d}[swap]{f}
&B \\
Q &
\end{tikzcd}
\end{center}
where $\iota$ is an embedding into the indecomposable act $B_S$ and $f$ is a homomorphism. If $Q_S$ contains a zero element then $Q_S^{\theta}=Q_S$ and there is nothing to prove. Suppose that $Q_S$ does not have any zero and consider the injective $S-$act $Q \cup {\Theta}$. Since $Q \cup {\Theta}$ is injective there exists a homomorphism $\bar{f} \in \text{Hom}(B, Q^{\theta})$ commuting the following diagram.
\begin{center}
\begin{tikzcd}
A \arrow{r}{\iota} \arrow{d}[swap]{f}
&B \\
Q^{\theta} &
\end{tikzcd}
\end{center}
Since $B_S$ is indecomposable ${\rm{Im}}\bar{f} \subseteq \Theta$ or  ${\rm{Im}}\bar{f} \subseteq Q$. Now from $\bar{f}\iota=f$ and ${\rm{Codom}}f=Q$ we observe that the second case occurs. So $Q$ is an InC-injective $S-$act.
\end{proof}

Since the right $S-$act $S$ is indecomposable, as a result of the above proposition we reach the following corollary.

\begin{cor} \label{co4}
Every InC-injective $S-$act is weakly injective.
\end{cor}

The next corollary gives an analogous version of Skornjakov-Baer criterion for InC-injective acts.

\begin{cor} \label{co5}
A right $S$-act $Q$ is InC-injective if and only if it is injective relative to all inclusions into cyclic acts.
\end{cor}

\begin{proof}
As cyclic acts are indecomposable, the necessity part follows.

Conversely, suppose that $Q_S$ is injective relative to all inclusions into cyclic acts. In light of Theorem \ref{InC-inj+0=inj}, we should just prove that $Q^{\theta}$ is an injective $S-$act. It can be proved by the same way as the proof of the necessity part of Theorem \ref{InC-inj+0=inj}.
\end{proof}

\begin{remark}
Note that InC-injective acts in the class of non-injective acts are the closest ones to their envelopes. Besides, thanks to Lemma \ref{le3}, indecomposable components of injective acts are InC-injective, that is, injective acts are unions of InC-injective acts. The next example shows that the implication $injective \Longrightarrow InC-injective$ is strict.
\end{remark}

\begin{ex} \label{ex3}
Let $S$ be the free word monoid over $\{x\}$. Put $I=\{0,1\}$. Thus $I^S$ is a cogenerator. Set
$$ B=\{f\in I^S|~ fs ~~ \text{ is a zero element for some } s\in S \} $$
and \bdm A=\{f\in I^S| ~fs \text{ is not a zero element for each} ~s\in S\}. \edm Notice that all zero elements of $I^S$ belong to $B$. On the other hand  the mapping $f:S\longrightarrow I$  given by
$$f(x^i)= \left\{ \begin{array}{ll} 0 & \textrm{if $i$ is even}\\ 1 & \textrm{if $i$ is odd} \end{array} \right. $$
is an element of $I^S$ such that $fs$ is not a zero element for each $s\in S$. Hereby, $I^S$ is a disjoint union of nonempty sets $A$ and $B$. Since $S$ is commutative both $A$ and $B$ are subacts of $I^S$ for which $A$ does not have a zero element. Thus $A$ is a InC-injective $S-$act which is not injective.
\end{ex}

\begin{prop}\label{prod & coprod & retract of InC-injective}
For a monoid $S$ we have the following.
\begin{itemize}
\item[{\rm (i)}] Any retract of an InC-injective $S-$act is also InC-injective.


\item[{\rm (ii)}] All coproducts of InC-injective right $S-$acts are InC-injective if and only if $S$ is left reversible.
\end{itemize}
\end{prop}

\begin{proof}
The proof is the same as for injectivity (see for example \cite{kilp},  Proposition 1.7.30 and Proposition 3.1.13).
\end{proof}


\begin{prop} \label{pr4}
For a monoid $S$ the following are equivalent.
\begin{itemize}
\item[{\rm (i)}] All coproducts of injective acts are injective,
\item[{\rm (ii)}]  for some injective acts $A_S$ and $B_S$, $A_S \coprod B_S$ is (weakly) injective,
\item[{\rm (iii)}] for some acts $A_S$ and $B_S$, $A_S \coprod B_S$ is (weakly) injective, that is, there exists a decomposable (weakly) injective act,
\item[{\rm (iv)}]  $S$ is left reversible,
\item[{\rm (v)}]  subacts of indecomposable acts are indecomposable,
\item[{\rm (vi)}]  subacts of indecomposable injective acts are indecomposable,
\item[{\rm (vii)}]  any indecomposable injective $S-$act contains only one zero.
\end{itemize}
\end{prop}

\begin{proof}
The implications $i\Longrightarrow ii,$ $ii\Longrightarrow iii$ and $v\Longrightarrow vi$ are clear.

The implication $iv\Longrightarrow i$ can be seen in \cite[Proposition 3.1.13]{kilp}.

The equivalence of $iv$ and $v$ is thanks to Proposition \ref{pr7}.

$iii\Longrightarrow iv$ By contradiction, suppose that for some acts $A_S$ and $B_S$, $A_S \coprod B_S$ is (weakly) injective and $S$ is not left reversible. So there exist right ideals $I$ and $J$ of $S$ such that $I\cap J=\emptyset.$ Since $A_S \coprod B_S$ is  (weakly) injective, for some fix elements $a\in A$ and $b\in B$ the map $f: I\cup J \longrightarrow A \coprod B$ defined by $f(i)=ai$ for each $i\in I$ and $f(j)=bj$ for each $j\in J$, must be extended to a homomorphism $\overline{f}:S \longrightarrow A \coprod B.$ Definition of $f$ implies that $Im(\overline{f}) \cap A\neq \emptyset$ and $Im(\overline{f}) \cap B\neq \emptyset,$ meanwhile $S$ is indecomposable which provides that $Im(\overline{f}) \subseteq A$ and $Im(\overline{f}) \subseteq B,$ a contradiction.

To prove $vi \Longrightarrow vii$ let $Q$ be an indecomposable injective $S-$act with two zero elements $\Theta_1$ and $\Theta_2$. Our assumption implies that $\Theta_1 \sqcup \Theta_2$ is indecomposable, which is impossible.

$vii\Longrightarrow iv$.  On the contrary, suppose that $S$ is not left reversible. Then by proposition \ref{pr8}, there exists a set $I$ with $|I|\geq 2$ that $I^S$ is an indecomposable cofree $S-$act. Now it is clear that for different elements $x$ and $y$ in $I$, constant functions $f=c_{x}$ and $g=c_{y}$ in $I^S$ are two different zero elements of $I^S$ and the fact that $I^S$ is an injective $S-$act leads to a contradiction.
\end{proof}


In the next proposition we characterize monoids over which InC-injective acts are injective.

\begin{prop} \label{pr5}
For a monoid $S$ the following conditions are equivalent.
\begin{itemize}
\item[{\rm (i)}] All InC-injective acts are InD-injective (PInD-injective),
\item[{\rm (ii)}] all InC-injective acts are injective,
\item[{\rm (iii)}] injective envelop of indecomposable acts are indecomposable,
\item[{\rm (iv)}] all indecomposable injective acts are the same injective acts with only one zero or all injective acts are indecomposable,
\item[{\rm (v)}] $S$ is not a left reversible monoid or $S$ contains a left zero,
\item[{\rm (vi)}] indecomposable components of injective acts are injective,
\item[{\rm (vii)}] indecomposable components of injective acts are InD-injective.
\end{itemize}
\end{prop}
\begin{proof} $i \Longleftrightarrow ii$. If $Q$ is an InC-injective $S-$act, then it is InD-injective by assumption. Then by the statement after the definition of InD-injectivity $Q$ contains a zero and hence by Theorem \ref{InC-inj+0=inj} is injective. The converse is clear.

$ii\Longrightarrow iii$. Let $Q$ be an indecomposable $S-$act with an injective envelop $E(Q)=E$. Suppose, contrary to our claim that $E_S=A_S\sqcup B_S$. Since $Q$ is indecomposable, there is no loss of generality in assuming $Q_S \subseteq A_S$ which Lemma \ref{le3} implies that $A^{\theta}$ is injective. So $A_S$ is InC-injective and consequently is injective by our assumption which contradicts the minimality of $E_S$.

$iii\Longrightarrow iv$. Suppose that $S$ is a left reversible monoid. Regarding Proposition \ref{pr4} part $vii$, any indecomposable injective act contains only one zero. On the other hand suppose that $Q_S$ is an injective act with only one zero. On the contrary assume that $Q$ is decomposable and $\bigsqcup \limits _{i\in I} Q_i$ is its unique decomposition into indecomposable subacts. Since $Q$ contains only one zero element, there exists $i\in I$ such that  $Q_i$ does not have any zero. In light of  Lemma \ref{le3} $Q_i^{\theta }=Q_i \cup \{\theta\} $ is a decomposable injective act  and clearly is injective envelop of indecomposable right $S-$act $Q_i$, contrary to our assumption. Thus $Q_S$ is indecomposable.

If $S$ is not left reversible then all cofree right $S$-acts are indecomposable using Proposition \ref{pr8}. Now let $Q$ be an injective $S-$act. As for any $\emptyset \neq X$ there exists an $X$-cofree object in {\bf Act-}$S$, then $Q_S$ can be embedded into a cofree right $S-$act namely $I^S$ (see \cite[Proposition 2.4.3]{kilp}). Therefore $Q_S$ is a retract of indecomposable right $S-$act $I^S$ and hence is indecomposable.

$iv\Longrightarrow v$. Suppose, contrary to our claim, that $S$ is a left reversible monoid without a left zero. Since $S$ is left reversible the cofree right $S-$act $S^S$ is decomposable. Suppose $\bigsqcup \limits _{i\in I} Q_i$ is its unique decomposition into indecomposable acts. Let $\text{id}_S \in Q_{i_0},$ for some $i_0 \in I.$ Since $S$ does not have a left zero, $Q_{i_0}$ does not have a zero element. Indeed if $f$ is a zero of $Q_{i_0},$ then in account of Proposition \ref{pr6} and Proposition \ref{pr7} there exists $s, t\in S$ such that $fs=\text{id}_S t.$ So $\text{id}_S t$ is a zero element. So for each $x\in S, \text{id}_S tx=\text{id}_S t,$ which implies $tx=t,$ for each $x\in S.$ So $t$ is a left zero of $S$ which contradicts our assumption. Therefore $S^S$ is the disjoint union of subacts $A_S$ and $B_S$ where $A_S$ contains no zero element. Regarding Lemma \ref{le3}, $A_S \cup \{\theta\}$ is a decomposable injective $S-$act with only one zero which contradicts our assumption.

$v\Longrightarrow vi$. If $S$ is not a left reversible monoid, in the proof of the implication $iii\Longrightarrow iv$ we observe that all injective acts are indecomposable and we are done. If $S$ contains a left zero, clearly all indecomposable components of injective acts which are InC-injective acts contains a zero element and so is injective.

The implication $vi\Longrightarrow vii$ is Clear. As indecomposable components of injective acts are InC-injective, the equivalence of $i$ and $vii$ is resulted.
\end{proof}

\begin{prop} \label{all ind are InC-injective}
  All indecomposable right acts over a monoid $S$ are InC-injective if and only if all right $S-$acts are InC-injective or all indecomposable right $S-$acts are injective.
\end{prop}

\begin{proof}
Let all indecomposable right $S-$acts be InC-injective. If $S$ is left reversible, then by Proposition \ref{prod & coprod & retract of InC-injective} all right $S-$acts are InC-injective. If $S$ is not left reversible, then by Proposition \ref{pr5}, all InC-injective acts are injective. So by assumption all indecomposable right $S-$acts are injective. The converse is clear.
\end{proof}

 The following theorem for InC-injective acts is similar to the injective case. The proof is just an adaptation of the proof of \cite[Theorem 4.5.13]{kilp}.

\begin{thrm}
All right acts over a monoid $S$ are InC-injective if and only if $S$ is a regular principal right ideal monoid whose all idempotents are special.
\end{thrm}

\section{Indecomposable Domain Injective Acts}

In this section we study some properties of InD-injective acts and their relevance to the other injective properties. First we get some useful results.

\begin{prop}\label{LC-injective & locally cyclic}
For an act $C$ over a monoid $S$ the following statements are equivalent.

\begin{itemize}
\item[{\rm (i)}]  $C$ is InD-injective (PInD-injective),

\item[{\rm (ii)}] for any right $S-$act $B$ containing $C$, and any generating set $\{c_i\}_{i \in I}$ of an indecomposable subact of $C$, there exists an $S-$homomorphism $f \in \text{Hom}(B,C)$ such that $f(c_i)=c_i$, for any $i \in I,$

\item[{\rm (iii)}] for any $\{c_i\}_{i \in I}$ of elements of  $C$ that generates an indecomposable subact of $C$, there exists an $S-$homomorphism $f \in \text{Hom}(E(C),C)$ such that $f(c_i)=c_i$, for any $i \in I,$ where $E(C)$ is the injective envelope of $C.$
\end{itemize}
\end{prop}

\begin{proof}
(i)$\Longrightarrow$(ii)
Let $\{c_i\}_{i \in I}$ be a generating set of an indecomposable subact of $C$, say $D.$ Suppose that $B$ is a right $S-$act containing $C,$ and $i:D\longrightarrow B $ is the inclusion map. Since $C$ is InD-injective (PInD-injective), there exists an $S-$homomorphism $f \in \text{Hom}(B,C)$ extending the inclusion map $j\in \text{Hom}(D,C).$ So $fi=j$ which implies $f(c_i)=j(c_i)=c_i.$

(ii)$\Longrightarrow$(iii)
It suffices to put in (ii) $B:=E(C).$

(iii)$\Longrightarrow$(i)
Assume that $D$ is an indecomposable right $S-$act, $B$ is a right $S-$act containing $D$ and consider the following diagram:
\vspace{1.3 cm}
\begin{center}
\begin{tikzcd}
D \arrow[hook]{r}{i} \arrow{d}[swap]{h}
&B \\
 C  \arrow{d}[swap]{j} & \\
 E(C) &
\end{tikzcd}
\end{center}

where $i$ and $j$ are inclusions. Since $E(C)$ is injective, there exists $g\in \text{Hom}(B,E(C))$ extending $jh$ to $B.$ Since $D$ is indecomposable, $\{h(d)\}_{d\in D}$ generates an indecomposable subact of $C$. So our assumption yields the existence of  $f\in \text{Hom}(E(C),C)$ for which $f(h(d))=h(d),$ for each $d\in D.$ Now put $\overline{h}=fg$ from $B$ into $C.$ Then $\overline{h}(d)=f(g(d))=f(j(h(d))=f(h(d)),$ that is, $\overline{h}i=h.$ Thus $C$ is InD-injective (PInD-injective).
\end{proof}



\begin{prop} \label{Ind P.InD-injective is  injective}
Every indecomposable PInD-injective (InD-injective) right $S-$act is injective.
\end{prop}

\begin{proof}
Suppose that $A$ is an indecomposable PInD-injective right $S-$act. Let $E(A)$ be the injective envelope of the $S-$act $A$ and $i$ be the inclusion map from $A$ into $E(A).$ Since $A$ is PInD-injective, for the monomorphism, $\text{id}_A$ there exists a homomorphism $\overline{\text{id}_A} \in \text{Hom}(E(A),A)$ which extends $\text{id}_A,$ that is, $\overline{\text{id}_A}\circ i=\text{id}_A$. So $A$ is a retract of the injective $S-$act $E(A)$ and hence is injective.
\end{proof}

Regarding the Skornjakov-Baer criterion for injective acts the next result states a weaker criterion for InD-injective acts.

\begin{prop}\label{}
A right $S-$act $C$ is InD-injective (PInD-injective), if and only if $C$ has a zero and for each indecomposable subact $A$ of every indecomposable $S-$act $B$ and each $S-$homomorphism ($S-$monomorphism) $f\in \text{Hom}(A,C)$ there exists $g\in \text{Hom}(B,C)$ which extends $f.$
\end{prop}

\begin{proof}
In light of explanation after definition of InD-injectivity  we just need to prove sufficiency part.

Let $A$ be an indecomposable right $S-$act and $B$ a right $S-$act containing $A.$ Suppose that  $\coprod_{i\in I}B_i$ is the decomposition of $B$ into indecomposable subacts $B_i$ and $f\in \text{Hom}(A,C).$ Then $A\subseteq B_i,$ for some $i\in I.$ Since $B_i$ is indecomposable, by assumption there exists $\overline{f}\in \text{Hom}(B_i,C)$ which extends $f.$ Now define, $g: B \longrightarrow C$ by $g(x)=\overline{f}(x),$ for $x\in B_i$ and $g(x)=0,$ the zero of $C$, for $x\in B\setminus B_i.$ Then $g$ is a well-defined $S-$homomorphism from $B$ into $C$ extending $f$ as desired.

The proof for PInD-injectivity follows in the same manner.
\end{proof}


\begin{prop}\label{prod & coprod & retract of LC-inj and InD-inj}
For a monoid $S$ we have the following statements.
\begin{itemize}
\item[{\rm (i)}] Any retract of an InD-injective (PInD-injective) $S-$act is also InD-injective (PInD-injective).

\item[{\rm (ii)}] Let $\{A_i\mid i\in I\}$ be a family of right $S-$acts. Then $A=\prod_{i\in I}A_i,$  is InD-injective (PInD-injective) if and only if  $A_i$ is InD-injective (PInD-injective), for every $i\in I.$

\item[{\rm (iii)}] Let $\{A_i\mid i\in I\}$ be an arbitrary family of right $S-$acts. If  $A_i$ is InD-injective (PInD-injective), for every $i\in I,$ then $A=\coprod_{i\in I}A_i,$  is InD-injective (PInD-injective).
\end{itemize}
\end{prop}

\begin{proof}
Proofs of parts (i) and (ii) are the same as for injectivity (see for example \cite[Proposition 1.7.30 and  Proposition 3.1.12]{kilp}).
To prove (iii) assume that $\{A_i \mid i\in I\}$ is a family of right $S-$acts and each $A_i$ is InD-injective (PInD-injective). Let $A=\coprod_{i\in I}A_i$ and $D$ be an indecomposable right $S-$act, $B$ be a right $S-$act containing $D$ and $f\in \text{Hom}(D,A).$ Then $f(D)$ is an indecomposable subact of $A$ which implies $f(D)\subseteq A_j,$ for some $j\in I,$  that is, we may consider $f$ as a homomorphism into $A_j.$ Since $A_j$ is InD-injective (PInD-injective), there exists $\overline{f}\in \text{Hom}(B,A_j)$ which extends $f.$ It may be considered as a homomorphism from $B$ into $A$ extending $f.$ So $A$ is InD-injective (PInD-injective).
\end{proof}

By part (iii) of the above proposition, $\Theta_S \sqcup \Theta_S$ is InD-injective, but by \cite[Proposition 3.1.13]{kilp} it is not injective unless $S$ is left reversible. So the implication $\text{injectivity} \Longrightarrow \text{InD-injectivity}$ is strict.

The next proposition gives conditions under which the converse of part (iii) of the above proposition is valid.

\begin{prop}
InD-injective and PInD-injective properties are transferred from coproducts to their components  if and only if $S$ is not a left reversible monoid or $S$ contains a left zero.
\end{prop}

\begin{proof}
First suppose that InD-injective and PInD-injective properties are transferred from coproducts to their components. So indecomposable components of injective acts are InD-injective (PInD-injective) and hence injective by Proposition \ref{Ind P.InD-injective is  injective}. So by Proposition \ref{pr5}, $S$ is not a left reversible monoid or $S$ contains a left zero.

Now assume that $S$ is not a left reversible or contains a left zero. Suppose that $\{A_i\mid i\in I\}$ is a family of right $S-$acts and $A=\coprod_{i\in I}A_i$ is InD-injective (PInD-injective). First we show that each $A_i$ contains a zero element. If $S$ contains a left zero, then we are done. So assume that $S$ is not left reversible. Then there exist $a, b\in S$ such that $aS \cap bS=\emptyset.$ Let $\rho$ be the Rees congruence on $S$ induced by the right ideal $bS$ and consider the following diagram of right $S-$acts and $S-$homomorphisms:
\begin{center}
\begin{tikzcd}
aS \arrow{r}{\iota} \arrow{d}[swap]{g}
&S/\rho \\
 A_j  \arrow{d}[swap]{\tau} \\
 A
\end{tikzcd}
\end{center}






where $\iota$ and $\tau$ are inclusions and $g: aS \rightarrow A_j$ is defined by $g(as)=xas, s\in S$ for a fixed element $x\in A_j$. Since $aS$ is indecomposable and $A$ is InD-injective (PInD-injective), there exists $\overline{g}: S/ \rho \rightarrow A$ extending $g.$ Now indecomposability of $S/ \rho$, gives $\overline{g}(S/\rho) \subseteq A_i,$ for some $i\in I.$ As $\overline{g}$ extends to $g,$ $i=j$ and therefore $A_j$ contains a zero. Thus each $A_j$ is a retract of $A$ and is InD-injective (PInD-injective) by Proposition \ref{prod & coprod & retract of LC-inj and InD-inj} part (i).
\end{proof}

Recall that a right $S-$act $A$ is called \textit{quasi injective} if it is injective with respect to all inclusions from its subacts into $A,$ that is, $B=A$ in diagram \textbf{(I)}. Also a right $S-$act $A$ is called \textit{pseudo injective} if for any right $S-$act $C$ and any subact $B$ of $C$ any \textbf{monomorphism} from $B$ into $A$ can be extended to $C,$ that is, $f$ is a monomorphism in diagram \textbf{(I)}.

\begin{prop}\label{coproduct QI implies injective}
Let $A$ be a right $S-$act provided that $A\sqcup E(A)$ is pseudo injective (quasi injective). Then $A$ is an injective $S-$act.
\end{prop}

\begin{proof}
Let $A$ be a right $S-$act, $E(A)$ be its injective envelope and $A\sqcup E(A)$ be a pseudo injective $S-$act. Consider the following diagram of right $S-$acts and $S-$ homomorphisms.

\begin{center}
\begin{tikzcd}
A \arrow[hook]{r}{\imath} \arrow{d}[swap]{\text{id}}
& E(A) \arrow[hook]{r}{\jmath_2}
&A \sqcup E(A) \\
 A  \arrow{d}[swap]{\jmath_1} \\
 A\sqcup E(A)
\end{tikzcd}
\end{center}

where $\imath, \jmath_1$ and $\jmath_2$ are inclusions. Note that since $A\sqcup E(A)$ is pseudo injective, $A$ contains a zero element, says $\theta$. So we can define $p \in \text{Hom}(A\sqcup E(A),A)$ with

$$p(a)=\begin{cases} a & \text{if } a\in A \\ \theta & \text{otherwise.} \end{cases}$$

Then $p\jmath_1=\text{id}_A.$ Since $\jmath_1 \text{id}$ is monomorphism and $A\sqcup E(A)$ is pseudo injective so there exists $f\in \text{End}(A\sqcup E(A))$ such that $f\jmath_2\imath=\jmath_1\text{id}.$ Thus $(pf\jmath_2)\imath=pf\jmath_2\imath=p\jmath_1\text{id}=\text{id},$ that is, $A$ is a retract of $E(A)$. So $A$ is injective. The proof of quasi injectivity follows similarly.
\end{proof}

Note that in this proposition, $E(A)$  can be replaced with any injective $S-$act containing $A.$

\begin{thrm}\label{C-inj imply quasi inj}
Let $S$ be a monoid. The following statements are equivalent.
\begin{itemize}
\item[{\rm (i)}] All InD-injective  right $S-$acts are injective,
\item[{\rm (ii)}] all InD-injective right $S-$acts are quasi injective,
\item[{\rm (iii)}] $S$ is left reversible.
\end{itemize}
\end{thrm}

\begin{proof}

(i)$\Longrightarrow$ (ii). It is clear.


(ii)$\Longrightarrow$ (iii). Consider the one element $S-$act, $\Theta_S.$ Clearly $\Theta_S$ is InD-injective. Then Proposition \ref{prod & coprod & retract of LC-inj and InD-inj} implies that, $\Theta_S \sqcup \Theta_S \sqcup E(\Theta_S \sqcup \Theta_S)$ is InD-injective and consequently quasi injective, by assumption. So Proposition \ref{coproduct QI implies injective} results that $\Theta_S \sqcup \Theta_S$ is injective. Therefore $S$ is left reversible, by \cite[Proposition 3.1.13]{kilp}.

(iii)$\Longrightarrow$ (i). Let $S$ be a left reversible monoid and $C$ be an InD-injective right $S-$act. Since $C$ has a zero element, it suffices to prove that it is injective relative to all inclusions into cyclic right $S-$acts. Suppose that $B$ is a  cyclic right $S-$act, $A$ is a subact of $B$ and $f$ is a homomorphism from $A$ to $C.$ Since $B$ is indecomposable, $A$ is indecomposable thanks to Proposition \ref{pr7} and hence $f$ can be extended to $B$ by assumption. Therefore $C$ is injective.
\end{proof}

\begin{thrm}\label{homological classification of InD-injective ideals}
For a monoid $S$ the following statements are equivalent.

\begin{itemize}
\item[{\rm (i)}] All right ideals of $S$ are InD-injective,

\item[{\rm (ii)}] all indecomposable right ideals of $S$ are InD-injective,

\item[{\rm (iii)}] all right ideals of $S$ are PInD-injective,

\item[{\rm (iv)}] all indecomposable right ideals of $S$ are PInD-injective,

\item[{\rm (v)}] all indecomposable right ideals of $S$ are injective,

\item[{\rm (vi)}] $S$ is a regular self-injective monoid all of whose indecomposable right ideals are principal.

\end{itemize}
\end{thrm}

\begin{proof}
The implications (i)$\Longrightarrow$(ii) and (iii)$\Longrightarrow$(iv) are clear. Since every right ideal $I$ of $S$ is the coproduct of indecomposable right ideals each of which by assumption is PInD-injective, Proposition \ref{prod & coprod & retract of LC-inj and InD-inj} yields the implication (ii) $\Longrightarrow$ (iii). The implication (iv)$\Longrightarrow$(v)  follows from Proposition \ref{Ind P.InD-injective is  injective}.

(v)$\Longrightarrow$(vi) since all principal right ideals are indecomposable, by \cite[Theorem 4.5.10]{kilp}, $S$ is a regular self-injective monoid. Now by \cite[Proposition 3.5.5]{kilp}, every indecomposable right ideal of $S$ is generated by an idempotent.

(vi)$\Longrightarrow$(i) Since $S$ is regular self-injective monoid, all principal right ideals of $S$ are injective by \cite[Theorem 4.5.10]{kilp} and hence all indecomposable right ideals of $S$ are injective by assumption. Now let $I$ be an arbitrary right ideal of $S.$ Let $A$ be an indecomposable subact of a right $S-$act $B$ and $f\in \text{Hom}(A,I).$ Then $f(A)$ is an indecomposable right ideal of $S$ and so it is injective. So there exists $\overline{f}\in \text{Hom}(B,f(A)),$ extending $f$ as desired.
\end{proof}

\begin{thrm}\label{completely InD-injective monoid}
For a monoid $S$ the following statements are equivalent.

\begin{itemize}
\item[{\rm (i)}] All right $S-$acts are InD-injective,

\item[{\rm (ii)}] all indecomposable right $S-$acts are InD-injective,

\item[{\rm (iii)}] all right $S-$acts are PInD-injective,

\item[{\rm (iv)}] all indecomposable right $S-$acts are PInD-injective,

\item[{\rm (v)}] all indecomposable right $S-$acts are injective,

\item[{\rm (vi)}] $S$ is right absolutely injective,

\item[{\rm (vii)}] $S$ contains a right zero  and all its right ideals are generated by special idempotents.
\end{itemize}
\end{thrm}

\begin{proof}
The implications (i)$\Longrightarrow$(ii) and (iii)$\Longrightarrow$(iv) are clear and (iv)$\Longrightarrow$(v) follows from Proposition \ref{Ind P.InD-injective is  injective}. The implications (ii)$\Longrightarrow$(iii) and (v)$\Longrightarrow$(i) follows by adopting the proofs of  (ii)$\Longrightarrow$(iii) and (vi)$\Longrightarrow$(i) respectively in Theorem \ref{homological classification of InD-injective ideals}. The equivalence of (vi) and (vii) is \cite[Theorem 2]{Skornjakov69}. Now we show that (v) and (vi) are equivalent. The implication (vi)$\Longrightarrow$(v) is clear. For the converse, we claim that $S$ is left reversible. Two cases may occur.

\textbf{Case 1)} All products of indecomposable acts are indecomposable: Then by \cite[Proposition 3.4]{Mojtaba 7th},  $S$ is left reversible.

\textbf{Case 2)} There exists a family of indecomposable right $S-$acts, $\{ A_i \mid i\in I\}$ such that $A=\prod_{i\in I}A_i$ is decomposable: Then $A$ is injective by assumption. So Proposition \ref{pr4} implies that $S$ is left reversible.

Now every right $S-$act is the coproduct of its indecomposable components all of which are injective by assumption and is injective, since $S$ is left reversible.
\end{proof}


\begin{remark}
Note that Theorem \ref{completely InD-injective monoid} states that for a monoid $S$ if all indecomposable right $S-$acts are PInD-injective then all  right $S-$acts are injective. Therefore to check that $S$ is right absolutely injective it suffices to check the PInD-injectivity only for all indecomposable right $S-$acts, that is, in diagram \textbf{(I)} for injectivity we need only $A$ and $C$ to be indecomposable and $f$ to be a monomorphism.
\end{remark}

\end{document}